\documentclass[]{theclass} 
\usepackage{graphicx, xfrac, lineno, float, subcaption, tasks, comment, xcolor, booktabs, multirow}
\usepackage[T1]{fontenc}
\usepackage[normalem]{ulem}
\usepackage{datetime}
\usepackage{colortbl}
\usepackage{enumitem}
\usepackage{soul}

\begin{document}
\begin{frontmatter}

\titledata{Three-cuts are a charm:\\ acyclicity in 3-connected cubic graphs}
{The authors were partially supported by VEGA 1/0743/21, VEGA 1/0727/22 and APVV-19-0308.}

\authordata{Franti\v{s}ek Kardo\v{s}}
{LaBRI, CNRS, University of Bordeaux, Talence, F-33405, France; \\ Department of Computer Science, Faculty of Mathematics, Physics and Informatics\\ Comenius University, Mlynsk\'{a} Dolina, 842 48 Bratislava, Slovakia}
{frantisek.kardos@u-bordeaux.fr}{}
\authordata{Edita M\'{a}\v{c}ajov\'{a}}
{Department of Computer Science, Faculty of Mathematics, Physics and Informatics\\ Comenius University, Mlynsk\'{a} Dolina, 842 48 Bratislava, Slovakia}
{macajova@dcs.fmph.uniba.sk}{}
\authordata{Jean Paul Zerafa}
{St. Edward's College, Triq San Dwardu\\ Birgu (Citt\`{a} Vittoriosa), BRG 9039, Cottonera, Malta; \\ Department of Computer Science, Faculty of Mathematics, Physics and Informatics\\ Comenius University, Mlynsk\'{a} Dolina, 842 48 Bratislava, Slovakia; \\
Faculty of Economics, Management \& Accountancy, and Faculty of Education\\ L-Universit\`{a} ta' Malta, Msida, MSD 2080, Malta}
{zerafa.jp@gmail.com}
{}

\keywords{acyclicity, circuit, factor, perfect matching, cubic graph, snark}
\msc{05C15, 05C70}

\begin{abstract}
Let $G$ be a bridgeless cubic graph. In 2023, the three authors solved a conjecture (also known as the $S_4$-Conjecture) made by Mazzuoccolo in 2013: there exist two perfect matchings of $G$ such that the complement of their union is a bipartite subgraph of $G$. They actually show that given any $1^+$-factor $F$ (a spanning subgraph of $G$ such that its vertices have degree at least 1) and an arbitrary edge $e$ of $G$, there exists a perfect matching $M$ of $G$ containing $e$ such that $G\setminus (F\cup M)$ is bipartite. This is a step closer to comprehend better the Fan--Raspaud Conjecture and eventually the Berge--Fulkerson Conjecture. The $S_4$-Conjecture, now a theorem, is also the weakest assertion in a series of three conjectures made by Mazzuoccolo in 2013, with the next stronger statement being: there exist two perfect matchings of $G$ such that the complement of their union is an acyclic subgraph of $G$. Unfortunately, this conjecture is not true: Jin, Steffen, and Mazzuoccolo later showed that there exists a counterexample admitting 2-cuts. Here we show that, despite of this, every cyclically 3-edge-connected cubic graph satisfies this second conjecture. 
\end{abstract}

\end{frontmatter}

\section{Introduction}

In 2013, Giuseppe Mazzuoccolo \cite{MazzuoccoloS4} proposed three beguiling conjectures about bridgeless cubic graphs. His first conjecture, implied by the Berge--Fulkerson Conjecture \cite{BergeFulkerson}, is the following.

\begin{conjecture}[Mazzuoccolo, 2013 \cite{MazzuoccoloS4}]
Let $G$ be a bridgeless cubic graph. Then, there exist two perfect matchings of $G$ such that the complement of their union is a bipartite graph.
\end{conjecture}

This conjecture, which is no longer open, has been solved by the three authors. More precisely they prove the following stronger statement.

\begin{theorem}[Kardo\v{s}, M\'{a}\v{c}ajov\'{a} \& Zerafa, 2023 \cite{quelling1}]\label{theorem quelling 1}
Let $G$ be a bridgeless cubic graph. Let $F$ be a $1^+$-factor of $G$ and let $e\in E(G)$. Then, there exists a perfect matching $M$ of $G$ such that $e\in M$, and $G\setminus (F\cup M)$ is bipartite.
\end{theorem}

We note that a \emph{$1^+$-factor} of $G$ is the edge set of a spanning subgraph of $G$ such that its vertices have degree 1, 2 or 3. Theorem \ref{theorem quelling 1} not only shows the existence of two perfect matchings of $G$ whose deletion leaves a bipartite subgraph of $G$, but that for every perfect matching of $G$ there exists a second one such that the deletion of the two leaves a bipartite subgraph of $G$. In particular, Theorem \ref{theorem quelling 1} also implies that for every collection of disjoint odd circuits of $G$, there exists a perfect matching  which intersects at least one edge from each odd circuit (this was posed as an open problem by Mazzuoccolo and the last author in \cite{s4gmjp}, see also \cite{zerafa thesis}).

Mazzuoccolo moved on to propose two stronger conjectures, with Conjecture \ref{conjecture acyclic+} being the strongest of all three.

\begin{conjecture}[Mazzuoccolo, 2013 \cite{MazzuoccoloS4}]\label{conjecture acyclic}
Let $G$ be a bridgeless cubic graph. Then, there exist two perfect matchings of $G$ such that the complement of their union is an acyclic graph.
\end{conjecture} 

\begin{conjecture}[Mazzuoccolo, 2013 \cite{MazzuoccoloS4}]\label{conjecture acyclic+}
Let $G$ be a bridgeless cubic graph. Then, there exist two perfect matchings of $G$ such that the complement of their union is an acyclic graph, whose components are of order 2 or 3.
\end{conjecture}

Clearly, these last two conjectures are true for 3-edge-colourable cubic graphs, and Janos H\"agglund verified the strongest of these conjectures (Conjecture \ref{conjecture acyclic+}) by computer for all non-trivial snarks (non 3-edge-colourable cubic graphs) of order at most 34 \cite{MazzuoccoloS4}. 
However, 5 years later, Jin, Steffen, and Mazzuoccolo \cite{dmgt} gave a counterexample to Conjecture \ref{conjecture acyclic}. Their counterexample contains a lot of 2-edge-cuts and the authors state that the conjecture "could hold true for 3-connected or cyclically 4-edge-connected cubic graphs". In fact, as in real life, being more connected has its own benefits, and in this paper we show the following stronger statement.

\begin{theorem}\label{theorem acyclic}
Let $G$ be a cyclically 3-edge-connected cubic graph, which is not a Klee-graph. Then, for any $e\in E(G)$ and any $1^+$-factor $F$ of $G$, there exists a perfect matching $M$ of $G$ containing $e$ such that $G\setminus (F\cup M)$ is acyclic.
\end{theorem}

We 
remark that Klee-graphs (see Definition \ref{definition klee}), which are to be discussed further in Section \ref{section kleegraphs}, are 3-edge-colourable cubic graphs and so are not a counterexample to Conjecture \ref{conjecture acyclic}. However, the stronger statement given in Theorem \ref{theorem acyclic} does not hold for this class of graphs, and this is the reason why we exclude them. 


Although Theorem \ref{theorem acyclic} is not a direct consequence of the Berge--Fulkerson Conjecture, we believe that the results presented here and in \cite{quelling1} are valuable steps towards trying to decipher long-standing conjectures such as the Fan--Raspaud Conjecture \cite{FanRaspaud}, and the Berge--Fulkerson Conjecture itself.

In fact, we will prove the following statement, which is equivalent to Theorem \ref{theorem acyclic}.
\begin{theorem}\label{theorem acyclic2}
Let $G$ be a cyclically 3-edge-connected cubic graph, which is not a Klee-graph. Then, for any $e\in E(G)$ and any collection of disjoint circuits $\mathcal{C}$, there exists a perfect matching $M$ of $G$ containing $e$ such that every circuit in $\mathcal{C}$ contains an edge from $M$.
\end{theorem}

Indeed, given a collection of disjoint circuits $\mathcal{C}$, its complement is a $1^+$-factor, say $F_{\mathcal{C}}$. A perfect matching $M$ containing $e$ such that $G \setminus (F_{\mathcal{C}} \cup M)$ is acyclic must contain an edge from every circuit in $\mathcal{C}$. On the other hand, given a $1^+$-factor $F$, its complement is a collection of disjoint paths and circuits, and so it suffices to consider the collection $\mathcal{C}_F$ of circuits disjoint from $F$. A perfect matching $M$ containing $e$ such that every circuit in $\mathcal{C}_F$ contains an edge from $M$, clearly makes 
$G\setminus (F \cup M)$ acyclic. 

\subsection{Important definitions and notation}

Graphs considered in this paper are simple, that is, they cannot contain parallel edges and loops, unless otherwise stated. 

Let $G$ be a graph and $(V_1,V_2)$ be a partition of its vertex set, that is, $V_1\cup V_2=V(G)$ and $V_1\cap V_2=\emptyset$. Then, by $E(V_1,V_2)$ we denote the set of edges having one endvertex in $V_1$ and one in $V_2$; we call such a set an \emph{edge-cut}. An edge which itself is an edge-cut of size one is a \emph{bridge}. A graph which does not contain any bridges is said to be \emph{bridgeless}.

An edge-cut $X=E(V_1,V_2)$ is called \emph{cyclic} if both graphs $G[V_1]$ and $G[V_2]$, obtained from $G$ after deleting $X$, contain a \emph{circuit} (a 2-regular connected subgraph). The \emph{cyclic edge-connectivity} of a graph $G$ is defined as the smallest size of a cyclic edge-cut in $G$ if $G$ admits one; it is defined as $|E(G)|-|V(G)|+1$, otherwise. For cubic graphs, the latter only concerns $K_4$, $K_{3,3}$, and the graph consisting of two vertices joined by three parallel edges, whose cyclic edge-connectivity is thus 3, 4, and 2, respectively. An \emph{acyclic} graph is a graph which does not contain any circuits.

Let $G$ be a bridgeless cubic graph. A \emph{$1^+$-factor} of $G$ is the edge set of a spanning subgraph of $G$ such that its vertices have degree 1, 2 or 3. In particular, a \emph{perfect matching} and a \emph{$2$-factor} of $G$ are $1^+$-factors whose vertices have exactly degree 1 and 2, respectively. 

\section{Klee-graphs}\label{section kleegraphs}

\begin{definition}[\cite{klee dmtcs}]\label{definition klee}
A graph $G$ is a Klee-graph if $G$ is the complete graph on 4 vertices $K_4$ or there exists a Klee-graph $G_0$ such that $G$ can be obtained from $G_0$ by replacing a vertex by a triangle (see Figure \ref{fig:exklee}).
\end{definition}

\begin{figure}[ht]
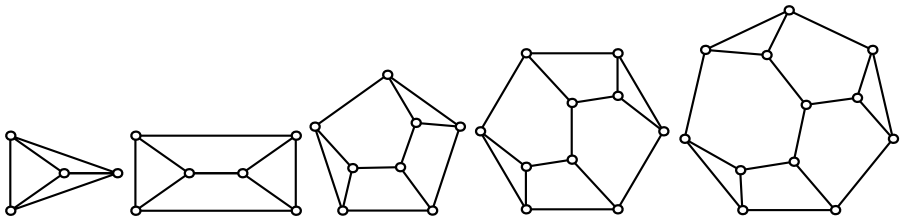

    \centering
    \includegraphics{klee.4}
    \includegraphics{klee.6}
    \includegraphics{klee.8}
    \includegraphics{klee.10}
    \includegraphics{klee.12}
    \caption{Examples of Klee-graphs on 4 upto 12 vertices, left to right.}
    \label{fig:exklee}
\end{figure}
For simplicity, if a graph $G$ is a Klee-graph, we shall sometimes say that $G$ is Klee. We note that there is a unique Klee-graph on 6 vertices (the graph of a 3-sided prism), and a unique Klee-graph on 8 vertices. As we will see in Section \ref{subsection klee}, these two graphs are Klee ladders, and shall be respectively denoted as $KL_6$ and $KL_8$. 

\begin{lemma}[\cite{klee dmtcs}]
The edge set of any Klee-graph can be uniquely partitioned into three pairwise
disjoint perfect matchings. In other words, any Klee-graph is 3-edge-colourable, and the colouring is unique up to a permutation of the colours.
\end{lemma}

Since Klee-graphs are 3-edge-colourable, they easily satisfy the statement of Conjecture \ref{conjecture acyclic}.

\begin{proposition}\label{prop klee}
Let $G$ be a Klee-graph. Then, $G$ admits two perfect matchings $M_1$ and $M_2$ such that $G\setminus (M_1\cup M_2)$ is acyclic.
\end{proposition}

The new graph obtained after expanding a vertex of a Hamiltonian graph (not necessarily Klee) into a triangle is still Hamiltonian, and so, since $K_4$ is Hamiltonian, all Klee-graphs are Hamiltonian. Hamiltonian cubic graphs have the following distinctive property.

\begin{proposition}\label{prop ham collection}
Let $G$ be a Hamiltonian cubic graph. Then, for any collection of disjoint circuits $\mathcal{C}$ of $G$ there exists a perfect matching $M$ of $G$ which intersects at least one edge of every circuit in $\mathcal{C}$.
\end{proposition}

\begin{proof}
Since $G$ is Hamiltonian, it admits three disjoint perfect matchings $M_1,M_2,M_3$ covering $E(G)$ such that at least two of them induce a Hamiltonian circuit. Without loss of generality, assume that $M_2\cup M_3$ induce a Hamiltonian circuit. Let $\mathcal{C}$ be a collection of disjoint circuits of $G$ for which the statement of the proposition does not hold. In particular, this implies that $M_1$ does not intersect all the circuits in $\mathcal{C}$ --- since the complement of $M_1$ is a Hamiltonian circuit, $\mathcal{C}$ consists of exactly one circuit. However, this means that $M_2$ (or $M_3$) intersects the only circuit in $\mathcal{C}$, contradicting our initial assumption.
\end{proof}

\begin{corollary}\label{cor klee collection}
For any collection of disjoint circuits $\mathcal{C}$ of a Klee-graph $G$ there exists a perfect matching $M$ of $G$ which intersects at least one edge of every circuit in $\mathcal{C}$.
\end{corollary}

On the other hand, we have to exclude Klee-graphs from Theorem \ref{theorem acyclic} (and Theorem \ref{theorem acyclic2}) since for some Klee-graphs there are edges contained in a unique perfect matching, as we will see in the following subsection.

\subsection{Other results about Klee-graphs}\label{subsection klee}

\begin{lemma}[\cite{klee dmtcs}]\label{lemma 2 disjoint triangles klee}
Let $G$ be a Klee-graph on at least 6 vertices. Then, $G$ has at least two triangles and all its triangles are vertex-disjoint.
\end{lemma}

Indeed, expanding a vertex into a triangle can only destroy triangles containing the vertex to be expanded. 

We will now define a series of particular Klee-graphs, which we will call \emph{Klee ladders}. Let $KL_4$ be the complete graph on 4 vertices, and let $u_4v_4$ be an edge of $KL_4$. For any even $n\ge 4$, let $KL_{n+2}$ be the Klee-graph obtained from $KL_n$ by expanding the vertex $u_n$ into a triangle. In the resulting graph $KL_{n+2}$, we denote the vertex corresponding to $v_n$ by $v_{n+2}$, and denote the vertex of the new triangle adjacent to $v_{n+2}$ by $u_{n+2}$.

In other words, the graph $KL_{2k+2}$ consists of the Cartesian product $P_2 \square P_k$ (where $P_t$ denotes a path on $t$ vertices) with two additional vertices $u_{2k+2}$ and $v_{2k+2}$ adjacent to each other, such that $u_{2k+2}$ ($v_{2k+2}$) is adjacent to the two vertices in the first (last, respectively) copy of $P_2$ in $P_2 \square P_k$ (see Figure \ref{fig:kleeLadder}).

Klee ladders can be used to illustrate why we have to exclude Klee-graphs from our main result. For a given Klee ladder $G$ there exists an edge $e$ such that $e$ is contained in a unique perfect matching of $G$, and therefore there is no hope for a statement like Theorem \ref{theorem acyclic2} to be true.

\begin{figure}[ht]
      \centering
      \includegraphics{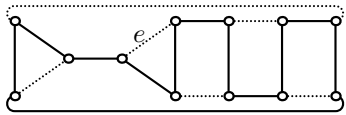}
      \caption{An example of a Klee ladder, $KL_{12}$. {There is a unique perfect matching (here depicted using dotted lines) containing the edge $e$.} The complement of this perfect matching is a Hamiltonian circuit.}
      \label{fig:kleeLadder}
\end{figure}


We will frequently use the following structural property of certain Klee-graphs.

\begin{lemma}\label{lemma 4 circuit klee}
Let $G$ be a Klee-graph on at least 8 vertices having exactly two (disjoint) triangles. Then, 
\begin{enumerate}[label=(\roman*)]
    \item exactly one edge of each triangle lies on a 4-circuit; and 
    \item if $G$ admits an edge joining the two triangles, then $G$ is a Klee ladder.
\end{enumerate}
\end{lemma}

\begin{proof}
We prove this by induction. Claim (i) is obvious for $KL_8$, the only Klee-graph on 8 vertices, so let $G$ be a Klee-graph on $n\ge 10$ vertices. By definition, it can be obtained from a smaller one, say $G_0$, by expanding a vertex into a triangle. Since $G$ only has two triangles, this operation must have destroyed a (single) triangle of $G_0$, which in turn gives rise to a 4-circuit containing exactly one of the edges of the new triangle. 

Moreover, if $G$ admits an edge $e$ joining the two triangles, then the corresponding edge $e_0$ in $G_0$ joins a triangle to a vertex contained in a (distinct) triangle, so it joins the two triangles of $G_0$. By induction, $G_0$ is a Klee ladder, say $KL_{n-2}$, for some $n\geq 8$, and the edge $e_0$ is the edge $u_{n-2}v_{n-2}$ (see the definition of Klee ladders above). Claim (ii) follows immediately.
\end{proof}

\section{Proof of Theorem \ref{theorem acyclic2}}

\begin{proof}
Let $G$ be a minimum counterexample to the statement of Theorem \ref{theorem acyclic2}. Since $K_4$ is Klee, $G$ has at least six vertices. There are only two 3-connected cubic graphs on six vertices, namely $KL_6$ and $K_{3,3}$. The former is Klee. For the latter, $K_{3,3}$, a collection of disjoint circuits can only contain one circuit on either four or six vertices and in both cases it is easy to check that every edge is contained in a perfect matching intersecting the prescribed circuit. Therefore, $G$ has at least eight vertices.

Let $e\in E(G)$ be an edge of $G$ such that there exists a collection of disjoint circuits such that for every perfect matching $M$ containing $e$ there is a circuit in the collection containing no edge from $M$. Amongst all such collections, we can choose an inclusion-wise minimal one, denoted by $\mathcal{C}$. 
By the choice of $\mathcal{C}$, we may assume that $e\notin C$ for any $C\in\mathcal{C}$. 

In the sequel, we will prove progressively a series of structural properties of $G$. Before that, we need to define three additional graph families. Let $KL_{2k-2}$ be the Klee ladder on $2k-2$ vertices with $k\ge 4$; let $u_{2k-2}$ and $v_{2k-2}$ be the two vertices contained in the two triangles, say $u_{2k-2}u_1u_2$ and $v_{2k-2}v_1v_2$, which are adjacent to each other. Moreover, we may assume that $KL_{2k-2} \setminus \{u_{2k-2},v_{2k-2}\}$ contains two disjoint paths of length $k-3$, one from $u_1$ to $v_1$ and the other from $u_2$ to $v_2$.

We remove the vertices $u_{2k-2}$ and $v_{2k-2}$ and replace them by four vertices, say $u'_1$, $u'_2$, $v'_1$, and $v'_2$, adjacent to $u_1$, $u_2$, $v_1$, and $v_2$, respectively, and we add a 4-cycle passing through the four new vertices. In fact, we can see the last operation as adding a complete graph on 4 vertices and removing a perfect matching. Up to symmetry, only three outcomes are possible.

\begin{itemize}
    \item A \emph{ladder} $L_{2k}$ is obtained if the edges $u'_1v'_2$ and $u'_2v'_1$ are missing.
    \item A \emph{M\"obius ladder} $ML_{2k}$ is obtained if the edges $u'_1v'_1$ and $u'_2v'_2$ are missing.
    \item A \emph{quasi-ladder} $QL_{2k}$ is obtained if the edges $u'_1u'_2$ and $v'_1v'_2$ are missing.
\end{itemize}
Observe that the ladder $L_{2k}$ is the graph of a $k$-sided prism. Observe that ladders and M\"obius ladders are vertex-transitive.\\

\begin{figure}[ht]
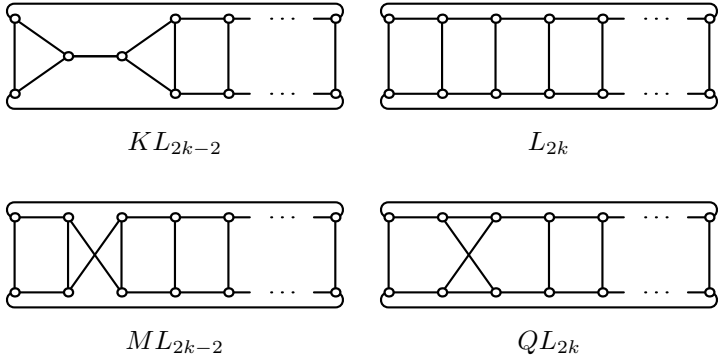

\centering
\begin{tabular}{cc}
     \includegraphics{ladders.0}&\includegraphics{ladders.1}  \\
     $KL_{2k-2}$     &  $L_{2k}$ \\
     & \\
     \includegraphics{ladders.2}&\includegraphics{ladders.3}  \\
     $ML_{2k-2}$     &  $QL_{2k}$ 
\end{tabular}
    \caption{An illustration of a Klee ladder, a ladder, a M\"obius ladder, and a quasi-ladder.}
    \label{fig:ladders}
\end{figure}

\noindent\textbf{Claim 1.} The graph $G$ is not a ladder, a M\"obius ladder, nor a quasi-ladder.

\noindent\emph{Proof of Claim 1.} 
We proceed by contradiction. Suppose that $G\in\{{L}_{2n}, {ML}_{2n}, QL_{2n}: n\geq 4\}$, and let $\mathcal{C}$ be a collection of disjoint circuits in $G$. 
We prove that for every edge $e$ there exists a perfect matching $M_e$ containing $e$ such that its complement is a Hamiltonian circuit, say $C_e$; moreover, there exists yet another perfect matching $M'_e$ containing $e$. The first perfect matching can be used to prove Theorem \ref{theorem acyclic2} unless $\mathcal{C}= \{C_e\}$. If this is the case, then we can use $M'_e$.

In most of the cases, the second perfect matching $M'_e$ can be obtained from $M_e$ by the following operation: we find a 4-circuit consisting of the edges $e_1, e_2, e_3, e_4$ (in this cyclic order) avoiding $e$ and containing exactly two edges from $M_e$, say $e_1$ and $e_3$. We then set $M'_e=(M_e \setminus \{e_1,e_3\}) \cup \{e_2,e_4\}$. In other words, $M'_e$ is obtained as the symmetric difference of $M_e$ and a suitable 4-circuit.

If $G$ is a ladder or a M\"obius ladder, then $G$ is vertex-transitive, and there are only two edge orbits. It suffices to distinguish between $e$ being an edge contained in two 4-circuits (vertical according to Figure \ref{fig:ladders}) or in a single one (horizontal or diagonal). An example of a pair of perfect matchings $M_e$ and $M'_e$ having the desired properties is depicted in Figure \ref{fig:ladder}.

\begin{figure}[ht]
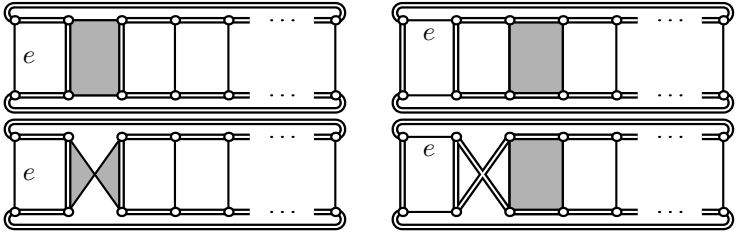

    \centering
    \includegraphics{ladders.10} $\quad$
    \includegraphics{ladders.11}
    \includegraphics{ladders.12} $\quad$
    \includegraphics{ladders.13}
    \caption{An example of a Hamiltonian circuit $C_e$ (drawn using double lines) avoiding a given edge $e$ whose complement is a perfect matching $M_e$ containing $e$, for both possible positions of the prescribed edge $e$ in a ladder (top line) or a M\"obius ladder (bottom line). A second perfect matching $M'_e$ can be obtained by the symmetric difference with the grey 4-circuit.}
    \label{fig:ladder}
\end{figure}

Let $G=QL_{2k}$ for some $k\ge 4$. If $e$ is an edge of the subgraph $P_2\square P_{k-2}$ or an edge of the 4-circuit $u'_1v'_1u'_2v'_2$ (see the definition of a quasi-ladder for the notation), then a pair of perfect matchings $M_e$ and $M'_e$ having the desired properties can be found in a same way as in the previous case, see Figure \ref{fig:quasiladder} for an illustration. 

Otherwise, let $e=u_1u'_1$ (for the remaining three edges the situation is symmetric). There is a unique Hamiltonian circuit $C_e$ avoiding $e$ and containing $u_2u'_2$, see Figure \ref{fig:quasiladder2} for an illustration. In this case, there is another perfect matching $M'_e$ containing $\{u_1u'_1,u_2u'_2,v_1v'_1,v_2v'_2\}$ and all the vertical edges of the subgraph $P_2\square P_{k-2}$ except for the first and the last one. \hfill {\tiny$\blacksquare$}\\

\begin{figure}[ht]
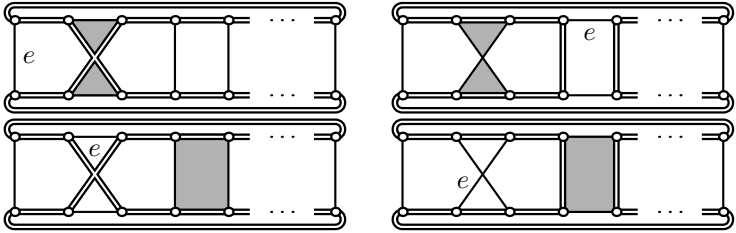

    \centering
    \includegraphics{ladders.30} $\quad$
    \includegraphics{ladders.31}
    \includegraphics{ladders.32} $\quad$
    \includegraphics{ladders.33}
    \caption{An example of a Hamiltonian circuit $C_e$ avoiding a given edge $e$ (drawn using double lines) whose complement is a perfect matching $M_e$ containing $e$, for an edge $e$ contained in the grid $P_2\square P_{k-2}$ (top line) and in the 4-circuit outside the grid (bottom line) of a quasi-ladder. A second perfect matching $M'_e$ can be obtained by the symmetric difference with the grey 4-circuit. }
    \label{fig:quasiladder}
\end{figure}

\begin{figure}[ht]
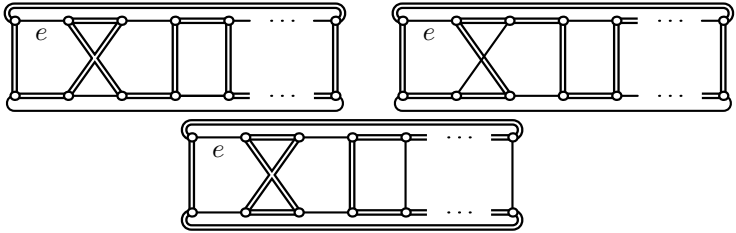

    \centering
    \includegraphics{ladders.34} $\quad$
    \includegraphics{ladders.35}
    \includegraphics{ladders.36} $\quad$
    
    \caption{An example of a Hamiltonian circuit $C_e$ avoiding a given edge $e$ (drawn using double lines) whose complement is a perfect matching $M_e$ containing $e$, for an edge $e$ joining a vertex in the grid $P_2\square P_{k-2}$ to a vertex of the 4-circuit outside the grid in a quasi-ladder (top line, two cases depending on the parity of the length of the grid). A second perfect matching $M'_e$ (bottom). }
    \label{fig:quasiladder2}
\end{figure}

\noindent\textbf{Claim 2.} The graph $G$ does not have any cyclic 3-edge-cuts.

\noindent\emph{Proof of Claim 2.} Suppose that $G$ admits a cyclic 3-edge-cut $E(V',V^{\prime\prime})$ with $E(V',V'')=\{f_{1}, f_{2}, f_{3}\}=:X$, where each $f_{i}=v'_{i}v''_{i}$, for some $v'_{1},v'_{2},v'_{3}\in V'$ and $ v''_{1},v''_{2}, v''_{3} \in V''$. Since $G$ has no 2-edge-cuts, the vertices $v'_{1},v'_{2},v'_{3},  v''_{1},v''_{2}, v''_{3}$ are all  distinct.

Either there is no circuit in $\mathcal{C}$ intersecting $X$, or the cut $X$ is intersected by a unique circuit $C_X$ in $\mathcal{C}$. Without loss of generality, we shall assume that when $C_X$ exists, $X\cap C_X=\{f_{2},f_3\}$.

Let $G'$ and $G''$ be the two graphs obtained from $G$ after deleting $X$ and joining the vertices $v'_{i}$ to a new vertex $v'$, and the vertices $v''_{i}$ to a new vertex $v''$. For each $i\in[3]$, let $e'_{i}=v'_{i}v'$ and $e''_{i}=v''_{i}v''$. 

\begin{figure}[ht]
      \centering
      \includegraphics[width=0.8\textwidth]{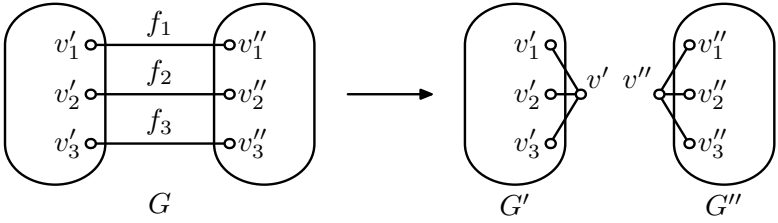}
      \caption{The graphs $G'$ and $G''$ when $G$ admits a cyclic 3-edge-cut $\{f_1,f_2,f_3\}$.}
      \label{figure 3cut}
\end{figure}

Let

\[\mathcal{C}' =
  \begin{cases}
  \{C\in \mathcal{C}\setminus \{C_X\} : C\cap E(G')\ne \emptyset\} \cup \{(C_X\cap E(G'))\cup\{e'_2,e'_3\}\}& $if $C_X$ exists,$\\
  \{C\in \mathcal{C} : C\cap E(G')\ne \emptyset\}  & $otherwise.$
  \end{cases}\]

Similarly, let
\[\mathcal{C}'' =
  \begin{cases}
  \{C\in \mathcal{C}\setminus \{C_X\} : C\cap E(G'')\ne \emptyset\} \cup \{(C_X\cap E(G''))\cup\{e''_2,e''_3\}\}& $if $C_X$ exists,$\\
  \{C\in \mathcal{C} : C\cap E(G'')\ne \emptyset\}  & $otherwise.$
  \end{cases}\]

It is not hard to see that $\mathcal{C}'$ ($\mathcal{C}''$) is a collection of disjoint circuits in $G'$ (in $G''$, respectively). Every circuit $C\ne C_X$ in $\mathcal{C}$ corresponds to a circuit either in $\mathcal{C}'$ or in $\mathcal{C}''$. The circuit $C_X$ (if it exists) corresponds to two circuits $C'_X$ and $C''_X$ in $\mathcal{C}'$ and $\mathcal{C}''$, respectively.

\textbf{Case A.} We first consider the case when $G$ does not admit any triangles, and claim that $G'$ (similarly $G''$) is not Klee. For, suppose that $G'$ is Klee. Since $G$ has no triangles, $|V(G')|\geq 6$, and so, by Lemma \ref{lemma 2 disjoint triangles klee}, $G'$ must admit two disjoint triangles. This is impossible since any triangle in $G'$ must contain the vertex $v'$. Hence, when $G$ does not admit any triangles, $G'$ and $G''$ are both not Klee.

Without loss of generality, we can also assume that at least one of the endvertices of $e$ corresponds to a vertex in $V'$. 
We consider two cases, depending on the existence of $C_X$. 

\emph{Case A1.} First, consider the case when $C_X$ does not exist. When $e\in X$, say $e=f_{1}$, then, by minimality of $G$, there exists a perfect matching $M'$ of $G'$ ($M''$ of $G''$) containing $e'_{1}$ ($e''_{1}$), intersecting every circuit in $\mathcal{C}'$ (in $\mathcal{C}''$, respectively). Consequently, $M=M'\cup M''\cup \{f_{1}\}\setminus\{e'_{1},e''_{1}\}$ is a perfect matching of $G$ containing $e=f_1$, intersecting every circuit in $\mathcal{C}$.

It remains to consider the case when $e\notin X$, and so the endvertices of $e$ both correspond to vertices in $G'$. Once again, for simplicity, we shall refer to this edge as $e$. Let $M'$ be a perfect matching of $G'$ containing $e$ intersecting every circuit in $\mathcal{C}'$. Without loss of generality, assume that $e'_{1}\in M'$. Let $M''$ be a perfect matching of $G''$ containing $e''_{1}$ intersecting every circuit in $\mathcal{C}''$. Let $M=M'\cup M''\cup \{f_{1}\}\setminus \{e'_{1},e''_{1}\}$. This is a perfect matching of $G$ containing $e$, intersecting every circuit in $\mathcal{C}$, a contradiction.

\emph{Case A2.} Suppose that $C_X$ exists. 
When $e\in X$, we have that $e=f_1$ by the choice of $C_X$, and so, by the minimality of $G$, there exists a perfect matching $M'$ of $G'$ ($M''$ of $G''$) containing $e'_{1}$ ($e''_{1}$), intersecting every circuit in $\mathcal{C}'$ (in $\mathcal{C}''$, respectively). Consequently, $M=M'\cup M''\cup \{f_{1}\}\setminus\{e'_{1},e''_{1}\}$ is a perfect matching of $G$ containing $e=f_1$.
Clearly, every circuit in $\mathcal{C}\setminus \{C_X\}$ is intersected by $M$. The circuit $C_X$ must be intersected by $M$ since $C'_X$ ($C''_X$) contains an edge of $M'$ ($M''$), not incident to $v'$ ($v''$, respectively). 

When $e\notin X$, the endvertices of $e$ both correspond to vertices in $G'$. Once again, for simplicity, we shall refer to this edge as $e$. Let $M'$ be a perfect matching of $G'$ containing $e$  intersecting every circuit in $\mathcal{C}'$. We have $e'_{i}\in M'$ for some $i\in[3]$. Let $M''$ be a perfect matching of $G''$ containing $e''_{i}$ intersecting every circuit in $\mathcal{C}''$. Let $M=M'\cup M''\cup \{f_{i}\}\setminus \{e'_{i},e''_{i}\}$. This is a perfect matching of $G$ containing $e$. As before, $M$ intersects every circuit in $\mathcal{C}$ unless $i=1$ and no edge of $G'$ or $G''$ corresponding to an edge of $C_X$ is in $M'$ or $M''$, which is impossible since $C'_X$ ($C''_X$) is a circuit in $\mathcal{C}'$ ($\mathcal{C}''$), so it contains an edge of $M'$ ( $M''$), not incident to $v'$ ($v''$, respectively).

\textbf{Case B.} What remains to be considered is the case when $G$ admits a triangle. Consequently, without loss of generality, we can assume that $G''$ is $K_4$. We note that in this case, $G'$ cannot be Klee because otherwise $G$ itself would be Klee. Thus, the inductive hypothesis can only be applied to $G'$ but not to $G''$. As in Case A, we can assume that at least one of the endvertices of $e$ corresponds to a vertex in $V'$, since if the endvertices of $e$ both belong to $V''$, say $e=v''_{i}v''_{j}$, a perfect matching of $G$ contains $e$ if and only if it contains $f_k$, where $\{i,j,k\}=[3]$.  We proceed as in Case A and note that the perfect matching $M'$ containing (the edge corresponding to) $e$ intersecting every circuit in $\mathcal{C}'$ obtained after applying the inductive hypothesis to $G'$ can be easily extended to a perfect matching $M$ of $G$ containing $e$. What remains to show is that $M$ intersects every circuit in $\mathcal{C}$. The only circuit possibly not intersected by $M$ is $C_X$, if it exists. However, this can only happen if $i=1$, and, if this is the case, then, in particular, $C'_X$ is a circuit in $G'$ and so contains an edge of $M'$ not incident to $v'$. This implies that $C_X$ contains the corresponding edge of $M$ in $G$, a contradiction.\hfill {\tiny$\blacksquare$}\\

\noindent\textbf{Claim 3.} The graph $G$ does not have any cyclic 4-edge-cuts.

\noindent\emph{Proof of Claim 3.} 
Suppose first that, in particular, $G$ has a 4-circuit $C=(v''_1,v''_2,v''_3,v''_4)$. Let $v'_1, v'_2, v'_3, v'_4$ be the vertices in $G-C$ respectively adjacent to $v''_1,v''_2,v''_3,v''_4$, let $f_i=v'_iv''_i$ for $i\in\{1,2,3,4\}$ and let $X = \{f_1, f_2, f_3, f_4\}$. The vertices $v'_i$ are pairwise distinct since $G$ does not have any cyclic 3-cuts. 

Let $\{i,j,k\}=\{2,3,4\}$. We denote by $G_{1i}$ the graph obtained after adding two new vertices $x$ and $y$ to $G-C$, such that:
\begin{itemize}
    \item $x$ and $y$ are adjacent; 
    \item $v'_1$ and $v'_i$ are adjacent to $x$; and
    \item $v'_j$ and $v'_k$ are adjacent to $y$.
\end{itemize} 

It is known that the graph $G_{1i}$ is 3-connected whenever $G$ is cyclically 4-edge-connected \cite{EKK}.
We claim that $G_{1i}$ is not Klee, for any $i\in\{2,3,4\}$. For, suppose not. Since $G$ does not admit any cyclic 3-cuts, by Lemma \ref{lemma 2 disjoint triangles klee}, the only two possible triangles in $G_{1i}$ are $(v'_1,v'_i,x)$ and $(v'_j,v'_k,y)$. Moreover, since $x$ is adjacent to $y$, by Lemma \ref{lemma 4 circuit klee}, $G_{1i}$ is a Klee ladder. For every $i\in\{2,3,4\}$, this implies that $G$ is a graph isomorphic to a ladder, a M\"obius ladder, or a quasi-ladder  --- this is a contradiction. 

We proceed by considering whether $e$ belongs to $E(C)$, $X$, or $E(G-C)$.

\textbf{Case A.} When $e\in E(C)$, then for every $i\in\{2,3,4\}$, every perfect matching of $G_{1i}$ containing $e'=xy$ extends to a perfect matching of $G$ containing $e$. The cut $X$ contains an even number of edges belonging to some circuit in $\mathcal{C}$. In particular, $E(C)$ can contain at most three circuit edges belonging to some circuit in $\mathcal{C}$, and so $C\not\in\mathcal{C}$.

If $X$ contains no circuit edges, then we can set $\mathcal{C}'=\mathcal{C}$ and apply induction on any $G'=G_{1i}$ to find a perfect matching $M'$ containing $e'$ intersecting every circuit in $\mathcal{C}'$, which readily extends to a perfect matching $M$ containing $e$ intersecting every circuit in $\mathcal{C}$. Note that the circuit $C$, in particular, is always intersected by at least one edge of $M$.

If there is a single circuit intersecting $X$ exactly twice, say $C_X$ passing through the edges $f_j$ and $f_k$, then we apply induction on the graph $G'=G_{1i}$ where $\{1,i\}=\{j,k\}$ (if $1\in\{j,k\}$), or $|\{1,i,j,k\}|=4$ (otherwise). The circuit $C'_X$ in $G'$ corresponding to $C_X$ contains two edges both incident to either $x$ or $y$. Hence, if a perfect matching $M'$ containing $e'=xy$ intersects every circuit in $\mathcal{C'}=(\mathcal{C}\setminus \{C_X\}) \cup \{C_X'\}$, then it extends to a perfect matching containing $e$ intersecting every circuit in $\mathcal{C}$, since $C'_X$ contains an edge in $M'$ not incident to $x$ (nor $y$).

If there are two distinct circuits each intersecting $X$ twice, say $C_X$ passing through the edges $f_1$ and $f_2$, and $D_X$ passing through the edges $f_3$ and $f_4$, or if there is a single circuit intersecting $X$ four times, say $C_X$ passing through the vertices $v'_1,v''_1,v''_2,v'_2$,  and also $v'_3,v''_3,v''_4,v'_4$, then we can apply induction on the graph $G_{12}$ with $e'=xy$ just like in the previous case.

\textbf{Case B.} When $e\in X$, say $e=f_1$, then every perfect matching of $G'=G_{13}$ containing $e'=xu_1$ extends to a perfect matching of $G$ containing $e$ in a unique way. If there is no circuit in $\mathcal{C}$ intersecting $X$, then we can set $\mathcal{C}'=\mathcal{C}\setminus \{C\}$ and apply induction directly. If there is a circuit in $\mathcal{C}$ intersecting $X$, say $C_X$, then $|C_X\cap X|=2$. The corresponding circuit $C'_X$ in $G'$ is well-defined: it always contains $y$ and eventually also $x$ (when $C_X\cap X\neq\{f_2,f_4\}$). A perfect matching $M'$ in $G'$ containing $e'$ intersecting every circuit in $\mathcal{C'}=(\mathcal{C} \setminus \{C_X\})\cup \{C'_X\}$ intersects $C'_X$ at a cut-edge incident to $y$ or an edge of $G-C$. In both cases, the corresponding perfect matching $M$ in $G$ containing $e$ intersects every circuit in $\mathcal{C}$, since $M$ intersects $C_X$ at an edge in $X$ or an edge of $G-C$.  

\textbf{Case C.} It remains to consider the case when $e\in E(G-C)$. Let $G'=G_{13}$ and let $e'$ be the edge of $G'$ corresponding to $e$ in $G$. Every perfect matching $M'$ of $G'$ containing $e'$ and not containing $xy$ extends to a perfect matching $M$ of $G$ containing $e$ in a unique way; every perfect matching $M'$ of $G'$ containing $e'$ and $xy$ extends to a perfect matching $M$ of $G$ in two distinct ways, whose symmetric difference is the 4-circuit $C$. In all the cases, we obtain a perfect matching $M$ of $G$ containing at least one edge of $C$.

If $X$ contains no edges belonging to any circuit in $\mathcal{C}$, then we can set $\mathcal{C}'=\mathcal{C}\setminus \{C\}$ and apply induction directly. The circuit $C$ in particular (if it is in $\mathcal{C}$) is always intersected by at least one edge of $M$.

If there is a single circuit intersecting $X$ exactly twice, say $C_X$, passing through the edge $f_1$ and $f_i$ for some $i\in\{2,3,4\}$, then the corresponding circuit $C'_X$ in $G'$ is well-defined: it always contains $x$ and eventually also $y$ (when $C_X\cap X\neq\{f_1,f_3\}$). We can set $\mathcal{C}'=(\mathcal{C}\setminus \{C_X\})\cup \{C'_X\}$ and apply induction. If $M'$ contains an edge of $C'_X$ not incident to $x$ nor $y$, then $M$ contains an edge of $C_X$ not incident to any vertex of $C$. If $M'$ contains the edge $xy$, then  amongst the two possible extensions of $M'$ into $M$ we can always choose one that contains at least one edge of $C_X$. If $M'$ contains an edge incident to $x$ or to $y$ distinct from $xy$, then $M$ contains the corresponding edge in $X$. In all the cases, it is possible to extend a perfect matching $M'$ of $G'$ containing $e'$ and intersecting every circuit in $\mathcal{C}'$ into a perfect matching $M$ of $G$ containing $e$ and intersecting every circuit in $\mathcal{C}$. 

If there are two distinct circuits each intersecting $X$ twice, say $C_X$ passing through the edges $f_1$ and $f_2$ and $D_X$ passing through the edges $f_3$ and $f_4$, then we apply induction on $G'$ with $\mathcal{C}'=\mathcal{C}\setminus \{C_X,D_X\}$. If the perfect matching $M'$ containing $e'$ and intersecting every circuit in $\mathcal{C}'$ obtained by induction also contains $xy$, then we can choose $M$ to contain both $v''_1v''_2$ and $v''_3v''_4$, and so it intersects both $C_X$ and $D_X$ as well. If $M'$ does not contain $xy$, then $|M\cap \{f_1,f_2,f_3,f_4\}|=2$. If $M$ contains exactly one of $f_1$ and $f_2$ then it also contains one of $f_3$ and $f_4$, and so $M$ intersects both $C_X$ and $D_X$. If $\{f_1,f_2\}\subset M$, then $v''_3v''_4\in M$; similarly, if $\{f_3,f_4\}\subset M$, then $v''_1v''_2\in M$. In all the cases $M$ intersects both $C_X$ and $D_X$, as desired.

If there is a single circuit intersecting $X$ four times, say $C_X$ passing through $v'_1v''_1v''_2v'_2$ and also $v'_3v''_3v''_4v'_4$, then we can apply induction on the graph $G'$ with $\mathcal{C}'=\mathcal{C}\setminus \{C_X\}$ just like in the previous case.

%

\medskip 
From this point on we may assume that $G$ does not contain any $4$-circuits. In particular, for every cyclic $4$-edge-cut $E(V',V'')$ both sides have at least six vertices.
Suppose that $G$ admits a cyclic 4-edge-cut $E(V',V^{\prime\prime})$ with $E(V',V'')=\{f_{1}, f_{2}, f_{3}, f_{4}\} =:X$, where each $f_{i}=v'_{i}v''_{i}$, for some $v'_{1},v'_{2},v'_{3},v'_{4}\in V'$ and $ v''_{1},v''_{2}, v''_{3}, v''_{4}\in V''$. Since $G$ has no 3-edge-cuts, the vertices $v'_{1},v'_{2},v'_{3},v'_{4},  v''_{1},v''_{2}, v''_{3},v''_{4}$ are all  distinct.

We define graphs $G'_{1i}$ and $G''_{1i}$ for $i\in \{2,3,4\}$ analogously as in the previous part. We denote by $x'$ and $y'$ ($x''$ and $y''$) the two new vertices in $G'_{1i}$ (in $G''_{1i}$), and by $e'_1$, $e'_2$, $e'_3$, $e'_4$ ($e''_1$, $e''_2$, $e''_3$, $e''_4$) the edges of $G'_{1i}$ (of $G''_{1i}$, respectively) corresponding to $f_1$, $f_2$, $f_3$, $f_4$, respectively, for $i\in\{2,3,4\}$. These graphs are all 3-connected \cite{EKK}. None of these graphs can be a Klee-graph: if this was the case, it would have to be a Klee ladder on at least eight vertices, but there are no 4-circuits at all in $G$, so this is impossible.

Consider first the case when $e\in X$, say $e=f_1$. If there is a circuit $C_X$ in $\mathcal{C}$ intersecting $X$, then $e\notin C_X$ and $|C_X\cap X|=2$. We may assume that $C_X\cap X=\{f_2,f_3\}$.
We consider all the three graphs $G'_{12}$, $G'_{13}$, and $G'_{14}$ (and all the three graphs $G''_{12}$, $G''_{13}$, and $G''_{14}$) at the same time. The circuit $C_X$ (if it exists) corresponds to a circuit $C'_X$ ($C''_X$) in each of them in a natural way, covering either one or two vertices  amongst $x'$ and $y'$ ($x''$ and $y''$, respectively). If $C_X$ does not exist, we shall proceed in the same manner, but letting $C_X$, $C_X'$, and $C_X''$ be equal to $\emptyset$. We apply induction with $e'=e'_1$ ($e''=e''_1$) and $\mathcal{C}'=((\mathcal{C}\setminus \{C_X\})\cap E(G'_{1i})) \cup \{C_X'\}$ ($\mathcal{C}''=((\mathcal{C}\setminus \{C_X\})\cap E(G''_{1i}))\cup \{C_X''\}$, respectively). Let $M'_i$ ($M''_i$) be a perfect matching in $G'_{1i}$ ($G''_{1i}$) containing $e'$ ($e''$) intersecting every circuit in $\mathcal{C'}$ (in $\mathcal{C''}$, respectively). Every perfect matching  amongst $M'_2$, $M'_3$, and $M'_4$ contains exactly one edge $e'_k$ corresponding to a cut edge $f_k$ for some $k\in\{2,3,4\}$ (besides the edge $e'$ corresponding to $f_1$) and the three values of $k$ cannot all be the same for the three perfect matchings. The same thing holds for the other three perfect matchings $M''_2$, $M''_3$, and $M''_4$. Therefore, for some $k\in\{2,3,4\}$ there exist two perfect matchings $M'_i$ and $M''_j$ containing the edge $e'_k$ and $e''_k$, respectively. We can combine them together into a perfect matching $M$ containing $e$ and $f_k$, intersecting every circuit in $\mathcal{C}$. In particular, if $C_X$ exists, then it can only be avoided by $M$ if $k=4$, but then $M'_i$ ($M''_j$) cannot contain any edge of $C'_X$ ($C''_X$) incident to $x'$ or to $y'$ (to $x''$ or to $y''$), so it intersects $C'_X$ inside $G[V']$ ($C''_X$ inside $G[V'']$, respectively). Consequently, $M$ intersects $C_X$ inside $G[V']$ and $G[V'']$. 

Consider next the case where $e\notin X$. We may assume that $e\in G[V']$. 
Let $\mathcal{C}_X$ be the set of circuits in $\mathcal{C}$ intersecting $X$. We have $|\mathcal{C}_X|\le 2$, and even if there is a single circuit in $\mathcal{C}_X$, it may contain all four edges of $X$.
Let $\mathcal{C}'_0$ ($\mathcal{C}''_0$) be the set of circuits from $\mathcal{C}$ within $G[V']$ ($G[V'']$, respectively). 
Given $G'=G'_{1i}$ ($G''=G''_{1i}$) for some arbitrary $i\in\{2,3,4\}$, let $\mathcal{C}'_X$ ($\mathcal{C}''_X)$ be the set of circuits obtained from the subpaths of circuits in $\mathcal{C}_X$ contained in $G[V']$ (in $G[V'']$) by adding the necessary edges from $\{e'_1,e'_2,e'_3,e'_4\}$ (from $\{e''_1,e''_2,e''_3,e''_4\}$) and eventually also the edge $x'y'$ ($x''y''$, respectively), if needed. Observe that $|\mathcal{C}'_X|=2$ ($|\mathcal{C}''_X|=2$) is possible when $|\mathcal{C}_X|=1$, and vice-versa. Finally, let $\mathcal{C}'=\mathcal{C}'_0 \cup \mathcal{C}'_X$ and $\mathcal{C}''=\mathcal{C}''_0 \cup \mathcal{C}''_X$.

Let $e'$ be the edge in $G'$ corresponding to $e$ in $G$. By induction, we obtain a perfect matching $M'$ containing $e'$ intersecting every circuit in $\mathcal{C}'$. 

Consider first the case when $x'y'\in M'$. We apply induction to obtain a perfect matching $M''$ of $G''=G_{1i}$, for any $i\in\{2,3,4\}$, containing $x''y''$ intersecting every circuit in $\mathcal{C}''$. Then, $M=(M' \setminus \{x'y'\})\cup (M''\setminus \{x''y''\})$ is a perfect matching of $G$ containing $e$. It is easy to check that $M$ intersects every circuit in $\mathcal{C}'_0$ and in $\mathcal{C}''_0$; it remains to certify that $M$ intersects all the circuits in $\mathcal{C}_X$. If $|\mathcal{C}_X|\le 1$, then we choose $G''_{1i}$ in such a way that $x''y''$ does not belong to any circuit in $\mathcal{C}''_X$, and so $M''$ 
 contains at least one edge (not incident to $x''$ or $y''$) of every circuit in $\mathcal{C}''_X$, and so the circuit in $\mathcal{C}_X$ will contain at least one edge from $M$.
 If $|\mathcal{C}_X|=2$, then it suffices to choose $G''_{1i}$ in such a way that $\mathcal{C}''_X$ contains two distinct circuits (avoiding $x''y''$), and then each of them will contain at least one edge (not incident to $x''$ or $y''$) from $M''$, and thus each circuit in $\mathcal{C}_X$ will contain at least one edge from $M$, as desired.

It remains to consider the case when for every choice of $G'=G'_{1i}$, a perfect matching $M'_i$ of $G'$ containing $e'$ and intersecting every circuit in $\mathcal{C}'$ never contains the edge $x'y'$. 
 Without loss of generality, we may assume that for $G'_{12}$ the perfect matching $M'_2$ contains the edges $e'_1$ and $e'_3$. We then consider $G'_{13}$. Again, without loss of generality, the perfect matching $M'_3$ contains the edges $e'_1$ and $e'_2$. Finally, we apply induction on  $G''=G''_{14}$ with $e''=e''_1$. Every perfect matching $M''$ of $G''$ containing $e''$ contains either $e''_2$ or $e''_3$, so it can be combined with either $M'_2$ or $M'_3$ to give a perfect matching $M$ of $G$ containing $e$. We may assume that $e''_2\in M''$. 
 It is easy to check that such a perfect matching $M$ intersects all the circuits in $\mathcal{C}'_0$ and in $\mathcal{C}''_0$;  it remains to certify that $M$ intersects all the circuits in $\mathcal{C}_X$. The only circuit from $\mathcal{C}_X$ potentially not intersected by $M$ is the one containing the edges $f_3$ and $f_4$, say $C_X$. However, the corresponding circuits $C'_X$ and $C''_X$ in $\mathcal{C}'_X$ and $\mathcal{C}''_X$ (there is exactly one on each side) respectively contain an edge of $M'_3$ (not incident to $x'$ or $y'$) and an edge of $M''$ (not incident to $x''$ or $y''$). Therefore, $M$ intersects $C_X$ at least twice, which is more than what is desired. 
\hfill {\tiny$\blacksquare$}\\

From this point on we may assume that $G$ is cyclically 5-edge-connected. We now consider the edges at distance 2 from $e$ (distance measured as the length of a shortest path joining corresponding vertices in the line graph of $G$).\\

\noindent\textbf{Claim 4.} Let $f$ be an edge at distance 2 from $e$. Then, $f\notin C$ for any $C\in \mathcal{C}$. 

\noindent\emph{Proof of Claim 4.} 
We will use a procedure that transforms a cubic graph $G$ into a cubic graph $G'$ smaller than $G$, such that every perfect matching of $G'$ containing a certain edge can be extended into a perfect matching of $G$ containing the corresponding edge. 

This operation was already used by Voorhoeve \cite{voorhoeve} to study perfect matchings in bipartite cubic graphs and it is one of the main tools used for counting perfect matchings in general in \cite{EsperetLovaszPlummer}. This technique is also used by the three authors in \cite{quelling1} to prove Theorem \ref{theorem quelling 1}.

Let $f=uv$, let the neighbours of $u$ distinct from $v$ be $\alpha$ and $\gamma$, and let the neighbours of $v$ distinct from $u$ be $\beta$ and $\delta$. 
In particular, since $G$ is cyclically 5-edge-connected, these four vertices are all distinct and non-adjacent. Without loss of generality, we may assume that $\alpha$ is an endvertex of $e$.

\begin{figure}[ht]
      \centering
      \includegraphics[width=0.5\textwidth]{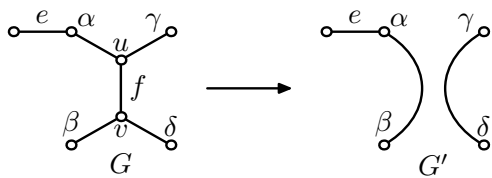}
      \caption{The vertices $\alpha,\beta, \gamma, \delta$ and an $(\alpha\beta:\gamma\delta)_{uv}$-reduction.}
      \label{figure reduction1}
\end{figure}

As shown in Figure \ref{figure reduction1}, we obtain a smaller graph by deleting the endvertices of $f$ (together with all edges incident to them) and adding the edges $\alpha\beta$ and $\gamma\delta$. Let this resulting graph be $G'$. We shall say that $G'$ is obtained after an $(\alpha\beta:\gamma\delta)_{uv}$-reduction. It is well-known that when applying this operation, the cyclic edge-connectivity of a cubic graph can drop by at most 2. Since $G$ is cyclically 5-edge-connected, $G'$ is cyclically 3-edge-connected. 

Let the edge in $G'$ corresponding to $e$, and the vertices in $G'$ corresponding to $\alpha,\beta, \gamma, \delta$ be denoted by the same name. We recall that any perfect matching of $G'$ which contains $e$ can be extended to a perfect matching of $G$ containing the edge $e$ (see also Figure \ref{figure reduction11}). In fact, let $M'$ be a perfect matching of $G'$ containing $e$. This is extended to a perfect matching $M$ of $G$ containing $e$ as follows:

\[M =
  \begin{cases}
  M'\cup\{u\gamma, v\delta\}\setminus \{\gamma\delta\}  & $if $\gamma\delta\in M'$,$ \\
  M'\cup \{f\} & $otherwise.$ 
  \end{cases}\]

\begin{figure}[ht]
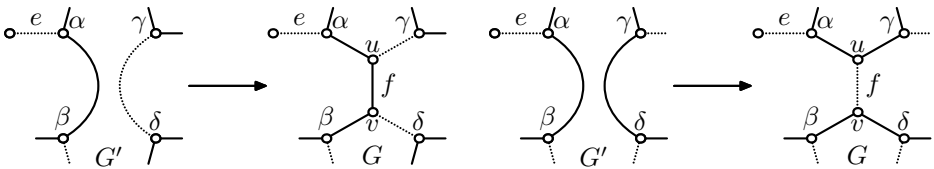

      \centerline{\hfill
      \includegraphics{reduction.2}
      \hfill\hfill
      \includegraphics{reduction.3}
      \hfill}
      \caption{Extending a perfect matching of $G'$ containing $e$ to a perfect matching of $G$ containing $e$. Dotted lines represent edges in $M$ or $M'$.} \label{figure reduction11}
\end{figure}

Suppose that for some edge $f$ at distance 2 from $e$, $f$ is in some circuit $C_f$ in $\mathcal{C}$. This means that exactly one of $u\alpha$ and $u\gamma$, and exactly one of $v\beta$ and $v\delta$ belong to $C_f$. Without loss of generality, we may assume that $u\alpha\in C_f$ if and only if $v\beta \in C_f$  (otherwise, we rename $\beta$ and $\delta$). Let $G'$ be the graph obtained from $G$ after an $(\alpha\beta:\gamma\delta)_{uv}$-reduction. 
Let $C'_f$ be the circuit in $G'$ corresponding to $C_f$ in $G$ obtained by replacing the 3-edge path passing through $u$ and $v$ by a single edge. Since $G$ is of girth 5, $C_f$ is a circuit of length at least 3.
Let $\mathcal{C}' = (\mathcal{C}\setminus\{C_f\})\cup\{C'_f\}$ be the collection of disjoint circuits of $G'$ obtained by this reduction. This is portrayed in Figure \ref{figure casei}.

\begin{figure}[ht]
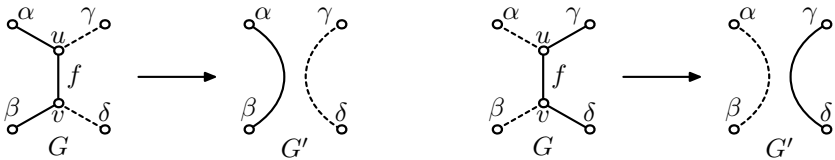

      \centering
      \hfill
      \includegraphics{reduction.4}\hfill\hfill
      \includegraphics{reduction.5}\hfill{}
            \caption{If $f\in C_f$, then we apply induction on $G'$ --- the graph obtained from $G$ after an $(\alpha\beta:\gamma\delta)_{uv}$-reduction. Dashed lines represent edges outside $\mathcal{C}$ or $\mathcal{C}'$, respectively.}
      \label{figure casei}
\end{figure}

Let us first assume that $G'$ is not a Klee-graph. Since $G'$ is cyclically 3-edge-connected and its order is strictly less than $G$, it is not a counterexample. Let $M'$ be a perfect matching of $G'$ containing $e$ intersecting all the circuits in $\mathcal{C}'$. We extend this perfect matching to a perfect matching $M$ of $G$ containing $e$ as described above (see Figure \ref{figure reduction11}), and claim that it intersects all the circuits in $\mathcal{C}$. Every circuit $C'\ne C'_f$ in $\mathcal{C'}$ is hit by an edge of $M'$ in $G'$, and so the corresponding circuit $C$ is hit by the corresponding edge of $M$ in $G$. The circuit $C'_f$ is hit by an edge $M'$ in $G'$, and so the corresponding circuit $C_f$ is hit by the corresponding edge in $G$, unless $\gamma\delta\in E(C_f)$ and the hitting edge is  $\gamma\delta$, but then $C_f$ is hit by both edges $\gamma u$ and $v\delta$. Observe that $M'$ cannot contain $\alpha\beta$ because $e\in M'$.

Therefore, $G'$ must be Klee. Since $G'$ is obtained after an $(\alpha\beta:\gamma\delta)_{uv}$-reduction, and $G$ is cyclically 5-edge-connected, by Lemma \ref{lemma 2 disjoint triangles klee}, the graph $G'$ admits exactly two (disjoint) triangles $T_{\ell}$ and $T_{r}$ such that $V(T_{\ell})=\{v_{\ell},\alpha,\beta\}$ and $V(T_{r})=\{v_{r},\gamma,\delta\}$, for some $v_{\ell}$ and $v_{r}$ in $G'$. Let $a,b,c,d$ be the vertices in $G'-\{v_{\ell},v_{r}, \alpha,\beta, \gamma,\delta\}$ which are adjacent to $\alpha,\beta,\gamma,\delta$, respectively (see Figure \ref{fig:toPetersen}). Furthermore, since $G$ is cyclically 5-edge-connected, by Lemma \ref{lemma 4 circuit klee} the edge $\alpha\beta$ ($\gamma\delta$) is the only edge in $T_{\ell}$ (in $T_{r}$) which lies on a 4-circuit. Therefore, $(\alpha,\beta,b,a)$ and $(\gamma,\delta, d, c)$ are 4-circuits in $G'$. Next we show that $a,b,c,d$ are pairwise distinct. Clearly, $a\neq b$, and $c\neq d$. Moreover, $a\neq c$, and $b\neq d$, otherwise $G$ would admit a 4-circuit. What remains to show is that $a\neq d$, and $b\neq c$. We first note that since $G$ is cubic, and $a=d$ if and only if $b=c$. Indeed, if $a=d$ and $b\neq c$, then, $a$ is adjacent to $\alpha, b, c, \delta$, a contradiction. Moreover, since $G$ is cyclically 4-edge-connected, if $a=d$, $G$ would be the Petersen graph. However, it is an easy exercise to check that the Petersen graph is not a counterexample.

\begin{figure}[ht]
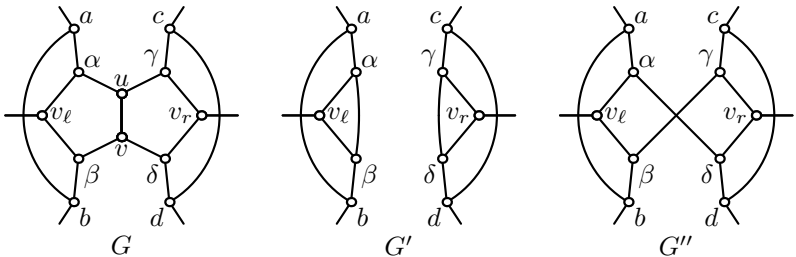

    \centering
    \includegraphics{last.1} $\quad$
    \includegraphics{last.2} $\quad$
    \includegraphics{last.3}
    \caption{If for some graph $G$, the graph $G'$ obtained after an $(\alpha\beta:\gamma\delta)_{uv}$-reduction is a Klee graph, then the graph $G''$ obtained after an $(\alpha\delta:\beta\gamma)_{uv}$-reduction is not.}
    \label{fig:toPetersen}
\end{figure}
Hence, $a,b,c,d$ are four distinct vertices. Consequently, if we apply an $(\alpha\delta:\beta\gamma)_{uv}$-reduction to $G$ we can be sure that $\alpha\delta$ and $\beta\gamma$ do not lie on a triangle in the resulting graph $G''$. In particular, $G''$ is not Klee. 
Let $\mathcal{C}''=\mathcal{C}\setminus \{C_f\}$. By the inductive hypothesis, $G''$ admits a perfect matching $M''$ containing $e$ intersecting every circuit in $\mathcal{C}''$.
By extending the perfect matching $M''$ to a perfect matching $M$ of $G$ containing $e$ as described above (see Figure \ref{figure reduction11}), we can deduce that $M$ intersects every circuit in $\mathcal{C}$, because, in particular, $M$ contains exactly one edge from $E(C_f)\cap \{u\gamma, uv, v\beta\}$. \hfill {\tiny$\blacksquare$}\\

From this point on we may assume that no edge $f$ at distance 2 from $e$ is contained in a circuit in $\mathcal{C}$. As a consequence, we have that no edge at distance at most 2 from $e$ is contained in a circuit in $\mathcal{C}$.\\

\noindent\textbf{Claim 5.} Every vertex at distance 2 from $e$ is traversed by a circuit in $\mathcal{C}$.

\noindent\emph{Proof of Claim 5.} Once again, let us consider an edge $f=uv$ at distance 2 from $e$, with vertices denoted $\alpha$, $\beta$, $\gamma$, $\delta$ as above. In particular, there can be a circuit in $\mathcal{C}$ passing through an endvertex of $f$ only if it is passes through the edges $\beta v$ and $v\delta$.

Suppose that there is no such circuit. As we have seen in the above claim, at least one of the the resulting graphs obtained by an $(\alpha\beta:\gamma\delta)_{uv}$-reduction or an $(\alpha\delta:\gamma\beta)_{uv}$-reduction is not Klee, and so, without loss of generality, we can assume that the graph $G'$ obtained after an $(\alpha\beta:\gamma\delta)_{uv}$-reduction is not Klee. In $G'$, there is a perfect matching $M'$ containing $e$ and intersecting every circuit in $\mathcal{C}'=\mathcal{C}$. It is easy to see that the perfect matching $M$ (of $G$) containing $e$ obtained as an extension of $M'$ still intersects all the circuits in $\mathcal{C}$, a contradiction. 
\hfill {\tiny$\blacksquare$}\\

As a consequence of Claim 5, we have the following.\\

\noindent\textbf{Claim 6.} The edge $e$ does not belong to a 5-circuit.

\noindent\emph{Proof of Claim 6.} Suppose that $e$ belongs to a 5-circuit $C=(t_1,t_2,t_3,t_4,t_5)$. Let the vertices in $G-V(C)$ which are adjacent to some vertex in $C$ be $v_1, v_2, v_3,v_4,v_5$, such that $v_i$ is adjacent to $t_i$ and $e=t_1t_2$. 
Since $G$ is cyclically 5-edge-connected, the $v_i$s are pairwise distinct. Moreover, by Claim 4, no edge in $C$ can be contained in a circuit of $\mathcal{C}$, but by Claim 5, the vertex $t_4$ must be traversed by a circuit of $\mathcal{C}$, which is clearly impossible. \hfill {\tiny$\blacksquare$}\\

We also show that $e$ cannot be at distance 2 from a 5-circuit.\\

\noindent\textbf{Claim 7.} Edges at distance 2 from $e$ do not belong to a 5-circuit.

\noindent\emph{Proof of Claim 7.} Suppose the above assertion is false and let $C=(t_1,t_2,t_3,t_4,t_5)$ be such a 5-circuit, with $t_1$ being an endvertex of an edge adjacent to $e$. We obtain a smaller graph $G'$ by deleting the edge $t_3t_4$, and smooth the vertices $t_3$ and $t_4$. Let the resulting graph be denoted by $G'$. It can be easily seen that $G'$ is cyclically 3-edge-connected and that it does not admit any triangles (and therefore note Klee), because otherwise, $G$ would contain 4-circuits. For each $i\in[5]$, let the vertex in $V(G)-V(C)$ adjacent to $t_i$ be denoted by $t'_i$. We proceed by first showing that a perfect matching $M'$ of $G'$ containing $e$ can be extended to a perfect matching $M$ of $G$ containing $e$. Without loss of generality, assume that $t_1t_2\in M'$. We extend this to a perfect matching of $G$ as follows: 
\[M =
  \begin{cases}
  M'\cup \{t_3t_4\} & $if $t_5t'_5\in M'$,$\\
  M'\cup \{t_1t_5,t_2t_3,t_4t'_4\} \setminus \{t_1t_2,t'_4t_5\} & $otherwise.$
  \end{cases}\]

Next, since in $G$, no edge at distance at most 2 from $e$ belongs to a circuit in $\mathcal{C}$, and every vertex at distance 2 from an endvertex of $e$ is traversed by one, we have that $t_2t_3$ and $t_4t_5$ belong to some circuit in $\mathcal{C}$ (possibly the same). We consider two cases depending on whether the edge $t_3t_4$ is in a circuit edge or not. 

\begin{enumerate}[label=(\roman*)]
\item When $t_3t_4$ is a circuit edge, then the vertices $t_2,t_3,t_4,t_5$ are consecutive vertices on some circuit $C_X$ in $\mathcal{C}$. In this case, we let 
$E(C'_X)=E(C_X)\cup \{t_1t_2,t_1t_5\}\setminus\{t_2t_3,t_3t_4,t_4t_5\}$
and $\mathcal{C}'=(\mathcal{C}\setminus\{C_X\})\cup \{C'_X\}$  to be a collection of disjoint circuits in $G'$. By the inductive hypothesis there exists a perfect matching $M'$ containing $e$ intersecting every circuit in $\mathcal{C}'$. The perfect matching $M$ of $G$ containing $e$ obtained from $M'$ as explained above clearly intersects every circuit in $\mathcal{C}$ (it contains either $t_2t_3$ or $t_3t_4$). This contradicts our initial assumption and so we must have the following case.

\item When $t_3t_4$ is not a circuit edge, we let $C_X$ and $C_Y$ (with $X$ not necessarily distinct from $Y$) to be the circuits containing the edges $t_2t_3$ and $t_4t_5$, respectively. The corresponding circuits $C'_X$ and $C'_Y$ in $G'$ are obtained by smoothing out the vertices $t_3$ and $t_4$. We then set $\mathcal{C}'=(\mathcal{C}\setminus \{C_X,C_Y\})\cup \{C'_X,C'_Y\}$. Note that the edges $t_2t'_2$ and $t_5t'_5$ belong to distinct circuits in $\mathcal{C}'$ if and only if they belong to distinct circuits in $\mathcal{C}$. By the inductive hypothesis, there exists a perfect matching $M'$ of $G'$ containing $e$ intersecting every circuit in $\mathcal{C}'$. Without loss of generality, assume that $t_1t_2\in M'$. The perfect matching $M'$ contains either $t_5t'_5$ or $t_5t'_4$, and consequently, so does the perfect matching $M$ obtained from $M'$ as explained above. This implies that $C_Y$ is intersected by $M$. If $C'_X\neq C'_Y$, then  $M'$ contains an edge of $C'_X$ not incident to $t_2$ in $G'$, and so $M$ contains the corresponding edge of $C_X$ in $G$. Altogether, the perfect matching $M$ obtained from $M'$ as shown above intersects every circuit in $\mathcal{C}$. This is again a contradiction to our initial assumption that $G$ is a counterexample --- thus proving our claim. \hfill {\tiny$\blacksquare$}
\end{enumerate}

Let's get back to analysing an edge $f=uv$ at distance 2 from $e$.
 We cannot use the reduction portrayed in Figure \ref{figure reduction1} as we do not have a guarantee that we can obtain a perfect matching $M$ intersecting the circuit in $\mathcal{C}$ containing the edges $v\beta$ and $v\delta$, which we shall denote by $C_v$. Since $G$ is cyclically 5-edge-connected, this latter circuit is of length at least 5. Let $\delta, v, \beta, y,z $ be consecutive and distinct vertices on this circuit (see Figure \ref{figure reduction2}). Moreover, let $w$ and $x$ be the vertices in $G$ respectively adjacent to $\beta$ and $y$, such that $w\beta, xy \notin E(C_v)$. We proceed by applying an $(\alpha\beta:\gamma\delta)_{uv}$-reduction followed by an $(\alpha x:wz)_{\beta y}$-reduction as portrayed in Figure \ref{figure reduction2}.

\begin{figure}[ht]
      \centering
      \includegraphics{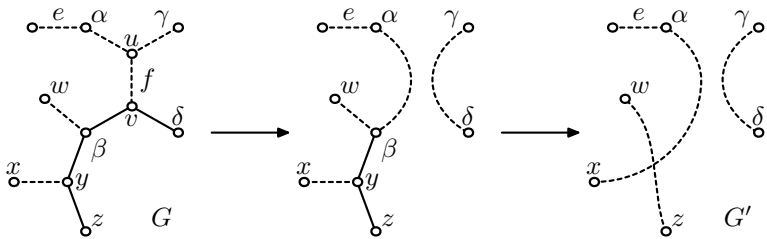}
      \caption{An $(\alpha\beta:\gamma\delta)_{uv}$-reduction followed by an $(\alpha x:wz)_{\beta y}$-reduction. Dashed edges represent edges outside $\mathcal{C}$ or $\mathcal{C'}$.}
      \label{figure reduction2}
\end{figure}

Let the resulting graph after these two reductions be denoted by $G'$, and let $\mathcal{C}'=\mathcal{C}\setminus\{C_v\}$.
Since $G'$ is obtained by applying twice the reduction at an edge at distance 2 from $e$, a perfect matching of $G'$ containing $e$ can always be extended to a perfect matching of $G$ containing $e$ (recall that $\beta y$ cannot be adjacent to $e$, since $\beta y\in C_v$). 
Moreover, any such matching contains either the edge $\beta y$ or the edge $yz$, and so it also contains at least one edge of the circuit $C_v$. Therefore, as long as $G'$ is cyclically 3-edge-connected and not Klee, by minimality of $G$, there exists a perfect matching $M'$ of $G'$ containing $e$ intersecting every circuit in $\mathcal{C}'$, which extends to a perfect matching $M$ of $G$ containing $e$ intersecting every circuit in $\mathcal{C}$. This contradicts our initial assumption that $G$ is a counterexample. 

Therefore, $G'$ is either Klee or admits a (cyclic) edge-cut of size at most 2 (more details after proof of Claim 8). \\

\noindent\textbf{Claim 8.} The graph $G'$ is not Klee.

\noindent\emph{Proof of Claim 8.} Suppose that $G'$ is Klee. The edge $\gamma\delta$ cannot be on a triangle otherwise the edges $uv$ and $u\gamma$ at distance 2 from $e$ belong to a 5-circuit, contradicting Claim 7. Since $G$ is cyclically 5-edge-connected, if $G'$ is Klee then we must have that $\alpha x$ and $wz$ each lie on a triangle (see Lemma \ref{lemma 2 disjoint triangles klee}). Therefore, in particular, if $G'$ is Klee, $\alpha$ and $x$ must have a common neighbour (in both $G'$ and $G$, so it is not $u$). 
This common neighbour cannot be $u'$, the neighbour of $\alpha$ not incident to $e$ and distinct from $u$, either, since $x$ would then be a vertex at distance 2 from the endvertex $\alpha$ of $e$ (via $u'$) and so, by Claim 5, it would be traversed by a circuit in $\mathcal{C}$. However, the edges $xu'$ and $xy$ are not in any circuit in $\mathcal{C}$, a contradiction.
Therefore, the common neighbour of $x$ and $\alpha$ is $\alpha'$, the other endvertex of $e$.
By Lemma \ref{lemma 4 circuit klee}, one edge of the triangle $(x,\alpha,\alpha')$ lies on a 4-circuit in $G'$, which is not present in $G$. First, consider the case when exactly one of $\alpha'x$ and $\alpha'\alpha$ lie on a 4-circuit, say $(\alpha', s,t,x)$ or $(\alpha', s,t,\alpha)$ accordingly. Since the edges $\alpha's, xt,\alpha'x,\alpha'\alpha, \alpha t$ all belong to $G$, $s$ and $t$ cannot be adjacent in $G$, and so $\{s,t\}$ is equal to $\{w,z\}$ or $\{\gamma,\delta\}$. If $\{s,t\}=\{\gamma,\delta\}$, then we either have that $\alpha'\gamma\in E(G)$, implying that $(\alpha',\gamma,u,\alpha)$ is a 4-circuit in $G$, or that $\alpha'\delta\in E(G)$, implying that $(\alpha',\alpha,u,v,\delta)$ is a 5-circuit in $G$ containing $e$, both a contradiction. Hence, $\{s,t\}=\{w,z\}$. Since, $G$ is cyclically 5-edge-connected, $x$ cannot be adjacent to $z$ nor $w$, and so, we must have that the edge lying on the 4-circuit with $w$ and $z$ is $\alpha'\alpha$. However, this is impossible since $z$ cannot be adjacent to an endvertex of $e$. Consequently, we must have that the edge of the triangle $(x,\alpha,\alpha')$ lying on a 4-circuit in $G'$ is $\alpha x$. In this case, $st$ cannot be an edge in $G$, otherwise $(\alpha', \alpha, s,t,x)$ would be a 5-circuit in $G$ containing $e$, contradicting Claim 6. Thus, $\{s,t\}$ is equal to $\{w,z\}$ or $\{\gamma,\delta\}$ once again. As before, $x$ cannot be adjacent to $z$ or $w$, implying that $\alpha$ being adjacent to $\gamma$ or $\delta$, respectively giving rise to $(\alpha,u,\gamma)$ or $(\alpha, u, v, \delta)$  in $G$, a contradiction. Therefore, $G'$ is not Klee.\hfill {\tiny$\blacksquare$}\\

Consequently, $G'$ must admit some (cyclic) edge-cut of size at most 2. Whenever $G$ is cyclically 4-edge-connected, by the analysis done at the end of the main result in \cite{quelling1}, we know the graph $G'$ is bridgeless, so if $G'$ is not (cyclically) 3-edge-connected, it admits a 2-edge-cut.
We next show that this cannot be the case, that is, if $G'$ admits a 2-edge-cut, then $G$ is not a counterexample to our statement.\\

\noindent\textbf{Claim 9.} $G'$ is cyclically 3-edge-connected, unless $\alpha$ is adjacent to $x$ in $G$.

\noindent\emph{Proof of Claim 9.}  Suppose that $G'$ admits a 2-edge-cut $X_2=\{g_1,g_2\}$. Let $\Omega_{1}=\{\alpha, \gamma,\delta, w,x,z\}$ and let $\Omega_{2}=\{u,v,\beta, y\}$.

We label the vertices of $G'\setminus X_2$ with labels $A$ and $B$ depending in which connected component of $G'\setminus X_2$ they belong to.
 Consequently, $G'$ has exactly two edges which are not monochromatic: $g_1$ and $g_2$. We consider different cases depending on the number of vertices in $\Omega_{1}$ labelled $A$ in $G'$, and show that, in each case, a 2-edge-cut in $G'$ would imply that $G$ is not cyclically 5-edge-connected or not a counterexample to our statement.  Without loss of generality, we shall assume that the number of vertices in $\Omega_{1}$ labelled $A$ is at least the number of vertices in $\Omega_{1}$ labelled $B$. We consider four cases.

\begin{enumerate}[label=(B\arabic*), start=0]
\item All the vertices in $\Omega_{1}$ are labelled $A$ in $G'$.

First, we extend this labelling of $V(G')$ to a partial labelling of $V(G)$ by giving to the vertices in $V(G)-\Omega_{2}$ the same label they had in $G'$. We then give label $A$ to all the vertices in $\Omega_{2}$. However, this means that $G$ has exactly two edges, corresponding to the edges in $X_2$, which are not monochromatic, a contradiction, since $G$ does not admit any 2-edge-cuts.
 
\item Exactly 5 vertices in $\Omega_{1}$ are labelled $A$ in $G'$.

This means that exactly one of the edges in $\{\alpha x, wz, \gamma\delta\}$ belongs to $X_2$, say $g_1$, without loss of generality. Once again, we extend this labelling to a partial labelling of $G$, and then give label $A$ to all the vertices in $\Omega_{2}$. However, this means that $G$ has an edge which has exactly one endvertex in $\Omega_{1}$ labelled $B$ and exactly one endvertex in $\Omega_{2}$ labelled $A$, which together with the edge $g_2$ gives a 2-edge-cut in $G$, a contradiction once again. 

\item Exactly 4 vertices in $\Omega_{1}$ are labelled $A$ in $G'$.

We consider two cases depending on whether there is one or three monochromatic edges in $\{\alpha x, wz,\gamma\delta\}$. First, consider the case when $\{\alpha x, wz,\gamma\delta\}$ has exactly one monochromatic edge, meaning that $X_2\subset\{\alpha x, wz,\gamma\delta\}$. As in the previous cases, we extend this labelling to a partial labelling of $V(G)$, and then give label $A$ to all the vertices in $\Omega_{2}$. However, this means that $G$ has exactly two edges each having exactly one endvertex in $\Omega_{1}$ labelled $B$ and exactly one endvertex in $\Omega_{2}$ labelled $A$, meaning that $G$ admits a 2-edge-cut, a contradiction. 

Therefore the edges $\{\alpha x, wz,\gamma\delta\}$ are all monochromatic: two edges with all their endvertices coloured $A$, and one edge with its endvertices coloured $B$. We extend this labelling to a partial labelling of $V(G)$, and then give label $A$ to all the vertices in $\Omega_{2}$. This gives rise to exactly two edges each having exactly one endvertex in $\Omega_{1}$ labelled $B$ and exactly one endvertex in $\Omega_{2}$ labelled $A$. These two edges together with the two edges in $X_2$ form a 4-edge-cut $X_4$ of $G$. Since the latter is cyclically 5-edge-connected this 4-edge-cut is not cyclic --- it separates two adjacent vertices from the rest of the graph. Since $w\ne z$ and $\gamma\ne \delta$ (otherwise there would be a 3-circuit in $G$) and also $\alpha \ne x$ (otherwise $C_v$ would contain edges at distance 1 from $e$), these two adjacent vertices in $G$ are endvertices of exactly one of $\alpha x$, $wz$, or $\gamma\delta$ in $G'$, and the 2-edge-cut in $G'$ separates a 2-circuit from the rest of the graph. However, $G$ can contain neither the edge $wz$ nor $\gamma\delta$, since $G$ has no 4-circuits. Thus, we must have that $\alpha$ is adjacent to $x$ in $G$. 

\item Exactly 3 vertices in $\Omega_{1}$ are labelled $A$ in $G'$.

Since $G'$ has exactly two edges which are not monochromatic, there is exactly one edge in $\{\alpha x, wz,\gamma\delta\}$ which is not monochromatic. The latter corresponds to one of the edges in $X_2$, say $g_1$, without loss of generality. As before, we extend this labelling to a partial labelling of $V(G)$, and then give label $A$ to all the vertices in $\Omega_{2}$. This gives rise to exactly three edges each having exactly one endvertex in $\Omega_{1}$ labelled $B$ and exactly one endvertex in $\Omega_{2}$ labelled $A$, which together with the edge $g_2$ from $X_2$ form a 4-edge-cut $X_4$ of $G$. As in the previous case, $X_4$ separates two adjacent vertices from the rest of the graph --- $G$ has exactly two vertices labelled $B$. As in the previous case, the endvertices of the monochromatic edge belonging to $\{\alpha x, wz,\gamma\delta\}$ (in $G'$) which are labelled $B$ must be either equal or adjacent in $G$, which is only possible if $\alpha$ is adjacent to $x$ in $G$.
\hfill {\tiny$\blacksquare$}
\end{enumerate}

Therefore, given any edge $f=uv$ at distance 2 from $e$, applying an $(\alpha\beta:\gamma\delta)_{uv}$-reduction followed by an $(\alpha x:wz)_{\beta y}$-reduction would lead to $\alpha$ being adjacent to $x$ in $G$. Let $y'$ and $z'$ be the two consecutive vertices on $C_v$ such that $y'$ is adjacent to $\delta$ (note that $y'$ or $z'$ are possibly equal to $z$). Let $w'$ and $x'$ respectively be the vertices in  $G-C_v$ adjacent to $\delta$ and $y'$, and let $\alpha'$ be the other endvertex of $e$. Applying an $(\alpha\delta:\gamma\beta)_{uv}$-reduction followed by an $(\alpha x':w'z')_{\delta y'}$-reduction leads to $\alpha$ being adjacent to $x'$. Therefore, $x'$ can be equal to $\alpha', u$ or $x$. If $x'=\alpha'$, then the edge $\delta y'$ would be an edge belonging to $C_v$  at distance 2 from the edge $e$, and if $x'=u$, then $(u,v,\delta,y')$ would be a 4-circuit in $G$, with both cases leading to a contradiction. Therefore, $x=x'$.

Let $G'$ be the graph obtained after an $(\alpha\beta:\gamma\delta)_{uv}$-reduction and let $\mathcal{C}'=\mathcal{C}\setminus \{C_v\}$. By induction, there exists a perfect matching $M'$ containing $e$ intersecting all the circuits in $\mathcal{C}'$, and it can be extended into a perfect matching $M$ of $G$ containing $e$. Suppose that $M$ does not intersect $C_v$ (it is the only circuit in $\mathcal{C}$ that $M$ could possibly avoid). To cover vertices $y$ and $y'$, we must have $\{xy,xy'\}\subset M$ --- which is impossible. 
\end{proof}

Here are some consequences of Theorem \ref{theorem acyclic2}. Corollary \ref{cor 3ec collection} follows by the above result and Corollary \ref{cor klee collection}.

\begin{corollary}\label{cor 3ec collection}
Let $G$ be a cyclically 3-edge-connected cubic graph and let $\mathcal{C}$ be a collection of disjoint circuits of $G$. Then, there exists a perfect matching $M$ such that $M\cap E(C)\neq \emptyset$, for every $C\in\mathcal{C}$.
\end{corollary}

\begin{corollary}
Let $G$ be a cyclically 3-edge-connected cubic graph. For every perfect matching $M_1$ of $G$, there exists a perfect matching $M_2$ of $G$ such that $G\setminus (M_1\cup M_2)$ is acyclic.
\end{corollary}

\end{document}